\documentclass[12pt]{amsart}
\usepackage{amsmath,amssymb,amsbsy,amsfonts,amsthm,latexsym,
                        amsopn,amstext,amsxtra,euscript,amscd,mathrsfs,color}
                        
\usepackage{float} 
\usepackage[english]{babel}
\usepackage{url}
\usepackage[colorlinks,linkcolor=blue,anchorcolor=blue,citecolor=blue]{hyperref}

\begin{document}

\newtheorem{theorem}{Theorem}
\newtheorem{lemma}[theorem]{Lemma}
\newtheorem{algorithm}{Algorithm}
\newtheorem{corollary}[theorem]{Corollary}
\newtheorem{proposition}[theorem]{Proposition}

\theoremstyle{definition}
\newtheorem{definition}[theorem]{Definition}
\newtheorem{question}[theorem]{Question}
\newtheorem{problem}[theorem]{Problem}
\newtheorem{example}[theorem]{Example}
\newtheorem{remark}[theorem]{Remark}
\newtheorem{conjecture}[theorem]{Conjecture}
\newtheorem{exercise}[theorem]{Exercise}

\newcommand{\comm}[1]{\marginpar{%
\vskip-\baselineskip 
\raggedright\footnotesize
\itshape\hrule\smallskip#1\par\smallskip\hrule}}


\def\cA{{\mathcal A}}
\def\cB{{\mathcal B}}
\def\cC{{\mathcal C}}
\def\cD{{\mathcal D}}
\def\cE{{\mathcal E}}
\def\cF{{\mathcal F}}
\def\cG{{\mathcal G}}
\def\cH{{\mathcal H}}
\def\cI{{\mathcal I}}
\def\cJ{{\mathcal J}}
\def\cK{{\mathcal K}}
\def\cL{{\mathcal L}}
\def\cM{{\mathcal M}}
\def\cN{{\mathcal N}}
\def\cO{{\mathcal O}}
\def\cP{{\mathcal P}}
\def\cQ{{\mathcal Q}}
\def\cR{{\mathcal R}}
\def\cS{{\mathcal S}}
\def\cT{{\mathcal T}}
\def\cU{{\mathcal U}}
\def\cV{{\mathcal V}}
\def\cW{{\mathcal W}}
\def\cX{{\mathcal X}}
\def\cY{{\mathcal Y}}
\def\cZ{{\mathcal Z}}

\def\C{\mathbb{C}}
\def\F{\mathbb{F}}
\def\K{\mathbb{K}}
\def\Z{\mathbb{Z}}
\def\R{\mathbb{R}}
\def\Q{\mathbb{Q}}
\def\N{\mathbb{N}}
\def\S{\mathbb{S}}
\def\T{\mathbb{T}}
\def\M{\textsf{M}}
\def\PP{\mathbb{P}}
\def\A{\mathbb{A}}
\def\p{\mathfrak{p}}
\def\n{\mathfrak{n}}
\def\X{\mathcal{X}}
\def\x{\textrm{\bf x}}
\def\w{\textrm{\bf w}}
\def\Nm{\mathrm{Nm}}
\def\Tr{\mathrm{Tr}}

\def\({\left(}
\def\){\right)}
\def\[{\left[}
\def\]{\right]}
\def\<{\langle}
\def\>{\rangle}

\def\gen#1{{\left\langle#1\right\rangle}}
\def\genp#1{{\left\langle#1\right\rangle}_p}
\def\genPs{{\left\langle P_1, \ldots, P_s\right\rangle}}
\def\genPsp{{\left\langle P_1, \ldots, P_s\right\rangle}_p}

\def\e{e}

\def\eq{\e_q}
\def\fh{{\mathfrak h}}

\def\lcm{{\mathrm{lcm}}\,}

\def\fl#1{\left\lfloor#1\right\rfloor}
\def\rf#1{\left\lceil#1\right\rceil}
\def\mand{\qquad\mbox{and}\qquad}

\def\jt{\tilde\jmath}
\def\ellmax{\ell_{\rm max}}
\def\llog{\log\log}

\def\ch{\hat{h}}
\def\GL{{\rm GL}}
\def\Orb{\mathrm{Orb}}
\def\vec#1{\mathbf{#1}}
\def\ov#1{{\overline{#1}}}
\def\Gal{{\rm Gal}}

\def\Crk{\mathrm{Crk}\,}

\numberwithin{equation}{section}
\numberwithin{theorem}{section}

\newcommand\al{\alpha}
\newcommand\be{\beta}
\newcommand\ga{\gamma}
\newcommand\lam{\lambda}

\title{Counting dynamical systems over finite fields}

\author{Alina Ostafe}
\address{School of Mathematics and Statistics, University of New South Wales, Sydney NSW 2052, Australia}
\email{alina.ostafe@unsw.edu.au}

\author{Min Sha}
\address{School of Mathematics and Statistics, University of New South Wales, Sydney NSW 2052, Australia}
\email{shamin2010@gmail.com}


\begin{abstract}
We continue previous work to count non-equivalent dynamical systems over finite fields generated by polynomials or rational functions. 
\end{abstract}

\maketitle

\section{Introduction}

\subsection{Motivation}

A \textit{(discrete) dynamical system} is simply a map, denoted by $(\S,f)$,  
$$
f: \S \to \S
$$
from a set $\S$ to itself, and its dynamics is the study of the behaviour of the points in $\S$ under iteration of the map $f$. For any integer $n\ge 0$, we denote by $f^{(n)}$ the $n$-th iteration of $f$ with $f^{(0)}$ denoting the identity map. For any $\al \in \S$, its orbit is defined by 
$$
\cO_f(\al)=\{\al, f(\al), f^{(2)}(\al),\ldots \}.
$$ 
The fundamental problem in the study of dynamics is to classify
the points of $\S$ according to the behaviour of their orbits. We refer to \cite{Schmidt,Silverman} for more about background on dynamical systems.

Choosing the set $\S$ as algebraic objects, like groups, number fields, $p$-adic fields and finite fields, yields the so-called \textit{algebraic dynamics}. They have many applications in computer science, cryptology, theoretical physics, cognitive science, and so on; see \cite{AnKh} for more details.

In this paper, continuing previous work~\cite{KLMMSS}, we study a novel question, that is, counting dynamical systems up to equivalence in some settings. 

First, we introduce the definition of equivalence of dynamical systems. 
\begin{definition}
The dynamical systems $(\S, f )$ and $(\T, g)$ are said to be  \textit{dynamically equivalent} if there exists a
bijection $\sigma : \S \to \T$ such that $
\sigma^{-1} \circ g \circ\sigma = f$.
\end{definition}

We can study the dynamical system $(\S,f)$ and comprehend ``dynamical equivalence'' from the viewpoint of graph theory. We define the \textit{functional graph} of $(\S, f )$ as a directed graph, denoted by $\cG_{(\S, f )}$ (or $\cG_{f}$ if $\S$ is fixed), with vertices at each element of $\S$, where there is an edge from $x$ to $y$ if and only if $f (x) = y$. So, the functional graph encodes the structure of the system $(\S,f)$. 
It is easy to see that dynamical equivalence coincides with isomorphism of functional graphs, and we will use both concepts interchangeably.

Bach and Bridy \cite{BaBr} estimated the number
of non-equivalent dynamical systems (or non-isomorphic functional graphs) generated by affine linear transformations of linear spaces over a finite field. In \cite{KLMMSS}, the authors obtained some theoretic estimates on the number of non-equivalent dynamical systems generated by all polynomials over a finite field of a given degree. See Section \ref{sec:previous} for precise statements. 

In this paper, except for reviewing previous results, the main objective is  applying the techniques in \cite{BaBr,KLMMSS} to explore more about counting non-equivalent dynamical systems generated by polynomials or rational functions over finite fields. 

In particular, in Section 3.1 we give an upper bound for the number of nonequivalent dynamical systems defined by sparse polynomials with fixed number of terms. 
We want to indicate that sparse (univariate or multivariate) polynomials are useful for several applications: pseudorandom number generators \cite{Bog}, hitting set generators \cite{Lu}, discrete logarithm over $\F_{2^n}$ \cite{Cop}, and efficient arithmetic in finite fields \cite{Doche}.

In Sections 3.2 and 3.3 we give the exact number of nonequivalent dynamical systems defined by very special classes of polynomials. We conclude the paper with treating the case of rational functions and posing some questions of possible interest.

\subsection{Convention and notation}

Given a dynamical system $(\S,f)$, a point $\al \in \S$ is called \textit{periodic} if $f^{(n)}(\al)=\al$ for some integer $n\ge 1$; the 
smallest such integer $n$ is called the \textit{period} of $\al$. If $f(\al)=\al$, then $\al$ is a \textit{fixed point}. Given a periodic point $\al$ of period $n$, the subgraph of the graph $\cG_{(\S,f)}$ with vertices at each element of the set $\{\al,f^{(1)}(\al),\ldots,f^{(n-1)}(\al)\}$ is called a \textit{cycle} of \textit{length} $n$. A point $\al$ is called \textit{preperiodic} if some iteration $f^{(n)}(\al)$ ($n\ge 0$) is periodic. Note that if $\S$ is a finite set, then every point is preperiodic. 

Let $\F_q$ be a finite field of $q$ elements, where $q=p^k$, $p$ is a prime number and $k$ is a positive integer, and we let $\F_q^*=\F_q\setminus \{0\}$. As usual, denote by  $(\F_q)^n$ the $n$-dimensional linear space over $\F_q$ for integer $n\ge 1$, and let $\GL_n(\F_q)$ be the general linear group of degree $n$ over $\F_q$. 
Besides, we use $M_n(\F_q)$ to denote the set of $n$-by-$n$ matrices with entries in $\F_q$.  

We use the Landau symbols $O$ and $o$ and the Vinogradov symbol $\ll$. We recall that the assertions $U=O(V)$ and $U\ll V$ are both equivalent to the inequality $|U|\le cV$ with some absolute constant $c$, while $U=o(V)$ means that $U/V\to 0$.

\section{Previous results}
\label{sec:previous}

Recall that an affine linear transform from $(\F_q)^n$ to itself has the form 
$$
f: (\F_q)^n \to (\F_q)^n, \quad f(x)=Ax+b,
$$
 where $A\in M_n(\F_q)$ and $b\in (\F_q)^n$. 
Denote by $D_q(n)$ the number
of non-equivalent dynamical systems (or non-isomorphic functional graphs) of affine linear transformations from $(\F_q)^n$ to itself. 
Bach and Bridy \cite[Theorem 1]{BaBr} showed that 
$$
\sqrt{n} \ll \log D_q(n) \ll \frac{n}{\log \log n}.
$$
The proof is based on the observation that given an affine linear transform $f$, 
for any affine automorphism $\phi$  the composition map  $\phi^{-1}\circ f \circ \phi$ has the same functional graph as $f$, that is, they generate the same dynamical system.  In addition, it is also an improvement on the well-known fact that the number of conjugacy classes in $\GL_n(\F_q)$ is less than $q^n$; for instance see \cite[Lemma A.1]{Maslen}.

Let $N_d(q)$ be the number of non-equivalent dynamical systems over $\F_q$ generated by all polynomials $f(X) \in \F_q[X]$
of degree $d\ge 2$ as the form 
$$
f: \F_q \to \F_q, \quad x \mapsto f(x).
$$
(In this paper, all the dynamical systems over $\F_q$ generated by polynomials follow this rule.)
 By counting the polynomials of degree $d$, one can easily get $N_d(q) \le (q-1)q^d$. 
In \cite{KLMMSS}, based on a similar observation as the above, the authors obtained some upper bounds concerning $N_d(q)$.
\begin{theorem}
\label{thm:Nd bound U}
For any $d\ge 2$ and $q$, we have
\begin{equation*}
\begin{split}
N_d(q)
&\le \left\{\begin{array}{ll}
q^{d-1}+(s-1)q^{d-1-\varphi(d-1)},& \text{if $p\nmid d$,}\\
q^{d-1}+(s-1)q^{d-1-\varphi(d-1)}+(q-1)q^{d/p-1},&\text{if $p \mid d$},
\end{array}
\right.
\end{split}
\end{equation*}
where $s=\gcd(q-1,d-1)$, and $\varphi$ is Euler's totient function.
In particular,  we have $N_d(q)\le 3q^{d-1}$.
\end{theorem}

Moreover, they also gave a lower bound for $N_d(q)$.

\begin{theorem}
\label{thm:Nd bound L}
Suppose that $\gcd(d-1,q)=1$. Then, for any $d\ge2$ and $e=\gcd(d,q-1) \ge 2$, we have
$$
N_d(q) \ge q^{\rho_{d,e} + o(1)}
$$
as $q \to \infty$, where
$$
\rho_{d,e}= \frac{1}{2(e- 1 + \log d/\log e)}.
$$
\end{theorem}

We also want to indicate that in \cite{KLMMSS} several algorithms are provided to list all the functional graphs up to isomorphism generated by polynomials of a given degree over finite fields.

\section{Main results}

\subsection{The case of sparse polynomials} 
Our first results give  upper bounds for the number of non-equivalent dynamical systems defined by arbitrary polynomials with fixed number of non-zero coefficients. 

Recall that $q=p^k$ for a positive integer $k$. Let $e_1,\ldots,e_s$ be distinct non-negative integers. We denote by $S_{e_1,\ldots,e_s}(q)$ the number of non-equivalent dynamical systems over $\F_q$ generated by all polynomials $f(X) \in \F_q[X]$ with $s$ non-zero terms of the form
\begin{equation}
\label{eq:sparse}
f=\sum_{i=1}^s a_iX^{e_i}.
\end{equation}

\begin{theorem}
\label{thm:sparse1}
For any integer $s\ge 1$, we have
$$
S_{e_1,\ldots,e_s}(q)\le (q-1)^{s-1}\gcd(e_1-1,\ldots, e_s-1,q-1).
$$
\end{theorem}

\begin{proof}
For $\lambda \in \F_q^*$, we define the bijection from $\F_q$ to itself 
$$
\psi_{\lam}:\  X \mapsto \lambda X
$$
with inverse $\psi_{\lam}^{-1}: \  X \mapsto \lambda^{-1}X$. Particularly, these bijections form a group of order $(q-1)$ in the usual way, which acts on the set of polynomials $f$ of the form~\eqref{eq:sparse} as the map
$$
f(X)\to \psi_{\lam}^{-1} \circ f \circ \psi_{\lam}(X).
$$

 The number of the orbits of the above group action can be calculated by the Burnside counting formula. This implies that 
$$
S_{e_1,\ldots,e_s}(q)\le \frac{1}{q-1} \sum_{\lambda\in\F_q^*} M_{e_1,\ldots,e_s}(\lambda),
$$
where $M_{e_1,\ldots,e_s}(\lambda)$ is the number of polynomials of the form~\eqref{eq:sparse} that are fixed by the above action. This reduces the problem to counting the number of coefficient vectors $(a_1,\ldots,a_s)\in\(\F_q^*\)^s$ such that $f(\lambda X)=\lambda f(X)$, and thus the number of solutions to $a_i\(\lambda^{e_i-1}-1\)=0$,  $i=1,\ldots,s$. As $a_i\ne 0$, $i=1,\ldots,s$, we have $\lambda^{e_i-1}=1$ for all $i=1,\ldots,s$. Each equation $\lambda^{e_i-1}=1$ has $\gcd(e_i-1,q-1)$ solutions in $\F_q^*$, and thus the number of $\lambda \in \F_q^*$ satisfying all the $s$ equations is $\gcd(e_1-1,\ldots,e_s-1,q-1)$.

As for each such $\lambda$ we have $M_{e_1,\ldots,e_s}(\lambda)= (q-1)^s$, putting everything together we obtain the desired result.
\end{proof}

We also use another approach to give a different bound, which does not depend on the exponents $e_1,\ldots,e_s$ and is better than Theorem \ref{thm:sparse1} in some special cases.

Let $\sigma$ be the automorphism of $\F_q$ which fixes $\F_p$ defined by $\sigma(x)=x^p$. 
For a polynomial $f\in\F_q[X]$ of the form~\eqref{eq:sparse}, 
we define
$$
\sigma(f)=\sum_{i=1}^s \sigma(a_i)X^{e_i}.
$$
Moreover, for $i\ge 1$, we have
$$
\sigma^i(f)=\sum_{i=1}^s \sigma^i(a_i)X^{e_i}.
$$

\begin{theorem}
\label{thm:sparse}
For any integer $s\ge 1$, we have
$$
S_{e_1,\ldots,e_s}(q)\le \frac{(q-1)^s}{k}+\frac{2(q^{1/2}-1)^{s}}{k}+(q^{1/3}-1)^{s}.
$$
\end{theorem}
\begin{proof}
It is easy to see that for any $0\le i\le k-1$, $f$ and $\sigma^{i}(f)$ define the same functional graph. Indeed, using the bijection from $\F_q$ to $\F_q$ defined by $x\to x^{p^{k-i}}$ and its inverse $x\to x^{p^i}$, the equivalence of the dynamical systems generated by $f$ and $\sigma^i(f)$ follows from
$$
X^{p^i}\circ f\circ X^{p^{k-i}}=\sigma^i(f)(X^q).
$$
We denote by
$$
\vec{a}=(a_1,\ldots,a_s)
$$
the vector of coefficients of $f$, and by $\sigma^i(\vec{a})=\(\sigma^i(a_1),\ldots,\sigma^i(a_s)\)$, $i=1,\ldots,k-1$.
Thus, all the vectors $\vec{a},\sigma(\vec{a}),\ldots,\sigma^{k-1}(\vec{a})$ define the same dynamical system. In other words, $S_{e_1,\ldots,e_s}(q)$ is upper bounded by the number of cycles of the map $\sigma$ on $\(\F_q^*\)^s$.

We denote by $d(\vec{a})$ the smallest degree field extension of $\F_p$ such that $\vec{a}\in\(\F_{p^{d(\vec{a})}}\)^s$ and by $\cN(d)$ the size of the set
$$
\{\vec{a}\in\(\F_q^*\)^s\mid d(\vec{a})=d\}.
$$
With this notation, note that for any such vector $\vec{a}$ and any integer $1\le i \le k-1$ we have $\sigma^i(\vec{a}) \in \(\F_{p^{d(\vec{a})}}\)^s$, then the number of cycles of the map $\sigma$ on $\(\F_q^*\)^s$ is at most $\sum_{d|k}\frac{\cN(d)}{d}$. 
Then, based on the discussion above, we have
\begin{equation*}
\begin{split}
S_{e_1,\ldots,e_s}(q)&\le \sum_{d|k}\frac{\cN(d)}{d}\le \sum_{d\mid k}\frac{\(p^{d}-1\)^s}{d}\\
&\le \frac{\(p^{k}-1\)^s}{k}+\frac{2\(p^{k/2}-1\)^s}{k}+\sum_{d\mid k, d\le k/3}\frac{\(p^{d}-1\)^s}{d}.
\end{split}
\end{equation*}
Now, as $p^s\ge 2$, we note that $\frac{\(p^{d}-1\)^s}{d}$ is an increasing function in $d\ge 1$, and thus we get
$$
\sum_{d\mid k, d\le k/3}\frac{\(p^{d}-1\)^s}{d}\le \sum_{1\le d\le k/3}\frac{\(p^{d}-1\)^s}{d}\le \(p^{k/3}-1\)^s.
$$
Putting everything together we get
$$
S_{e_1,\ldots,e_s}(q)\le \frac{\(p^{k}-1\)^s}{k}+\frac{2\(p^{k/2}-1\)^s}{k}+\(p^{k/3}-1\)^s,
$$
and thus we conclude the proof.
\end{proof}

We note that Theorem~\ref{thm:sparse} is better than Theorem~\ref{thm:sparse1} only when $e_1-1,\ldots, e_s-1,q-1$ have a large common factor. It would be certainly interesting to combine both types of bijections in Theorems~\ref{thm:sparse1} and~\ref{thm:sparse} to obtain a better estimate for $S_{e_1,\ldots,e_s}(q)$.

One can get a more explicit estimate in Theorem~\ref{thm:sparse} using the M\" obius inversion formula~\cite[Theorem 3.24]{LN}.  Indeed, as 
$$
\sum_{d|k}\cN(d)=\(p^k-1\)^s,
$$
applying the M\" obius inversion formula, we obtain
$$
\cN(k)=\sum_{d|k} \mu(d)\(p^{k/d}-1\)^s,
$$
where $\mu$ is the M\" obius function. Then, 
\begin{equation}
\label{eq:moeb}
S_{e_1,\ldots,e_s}(q)\le\sum_{d|k}\frac{\cN(d)}{d}=\sum_{e\mid k}\frac{\mu(e)}{e}\sum_{d\mid \frac{k}{e}}\frac{\(p^{d}-1\)^s}{d}.
\end{equation}

If for example $k$ is prime, then~\eqref{eq:moeb} gives a better estimate than Theorem~\ref{thm:sparse},
$$
S_{e_1,\ldots,e_s}(q)\le \frac{(q-1)^{s}}{k}+\frac{k-1}{k}(p-1)^{s}.
$$

We also note that Theorem~\ref{thm:sparse} holds also with any distinct integers $e_1,\ldots,e_s$, not necessarily non-negative.

\subsection{The case of linearised polynomials}
For integer $n\ge 1$, we denote by $L_{n}(q)$ the number of non-equivalent dynamical systems over $\F_q$ generated by all linearised polynomials of degree $p^n$ of the form
\begin{equation}
\label{eq:lin}
\cL (X)=\sum_{i=0}^{n}a_iX^{p^i}\in\F_q[X],\quad a_{n}\ne 0.
\end{equation}
We want to improve upon the trivial bound 
$$
L_n(q)< q^{n+1}.
$$

We follow exactly the same ideas as in the proof of~\cite[Theorem 1]{KLMMSS} to show the following nontrivial estimate. 

\begin{theorem}
\label{thm:lin}
For any integer $n\ge 1$, we have
$$
L_n(q) < (2p-2)q^{n-1} + 2q^{n-\varphi(n)},
$$
where $\varphi$ is Euler's totient function. In particular, we have 
$L_n(q) < 2pq^{n-1}$.
\end{theorem}

\begin{proof}
We use the same idea as in Theorem~\ref{thm:sparse1}.
For $\lambda \in \F_q^*$ and $\mu \in \F_q$, we define the bijection from $\F_q$ to itself 
\begin{equation}
\label{eq:phi}
\phi_{\lambda,\mu}:\  X \mapsto \lambda X + \mu
\end{equation}
with inverse $\phi_{\lambda,\mu}^{-1}: \  X \mapsto \lambda^{-1}(X - \mu)$. Particularly, these bijections form a group of order $(q-1)q$ in the usual way, which acts on the set of polynomials $\cL(X)$ of the form \eqref{eq:lin} as the map
$$
\cL(X) \to \phi_{\lambda,\mu}^{-1} \circ \cL \circ \phi_{\lambda,\mu}(X).
$$

As before, the number of the orbits of the above group action can be calculated by the Burnside counting formula. This implies that 
$$
L_n(q)\le \frac{1}{(q-1)q}\sum_{\lambda,\mu}M_n(\lam,\mu),
$$
where the sum runs over all pairs $(\lambda,\mu)\in\F_q^*\times \F_q$, and $M_n(\lambda,\mu)$ is the number of linearised polynomials of the form~\eqref{eq:lin} of degree $p^n$ fixed by the automorphism
$\phi_{\lambda,\mu}$ under the above group action. 

Simple computations show that the set of polynomials $\cL(X)$ of the form~\eqref{eq:lin} which are fixed by $\phi_{\lambda,\mu}$ are the polynomials that satisfy the conditions
\begin{equation}
\label{eq:cond1}
a_i(\lambda^{p^i}-\lambda)=0,\ i=0,\ldots,n,\quad \textrm{and}\quad  \cL(\mu)=\mu.
\end{equation}
In particular, as $a_n\ne 0$, we have $\lambda^{p^n}=\lambda$.
For fixed $\lambda,\mu$, we now count the coefficients $a_0,\ldots,a_n$ of $\cL(X)$ that satisfy the conditions~\eqref{eq:cond1}. 

Trivially, for $\lambda$ with $\lambda^{p^n}\ne \lam$, we have
$$
M_n(\lambda,\mu)=0.
$$

We consider first the case $\lambda\in\F_p^*$, that is $\lambda^p=\lam$. For $\mu=0$ we trivially have
$$
M_n(\lambda,0)=(q-1)q^n
$$
for $p-1$ values of $\lambda$. 
For $\mu\ne 0$, if we fix $a_1,\ldots,a_n$, the coefficient $a_0$ is uniquely defined by $\cL(\mu)=\mu$ in~\eqref{eq:cond1}, and thus one gets
$$
M_n(\lambda,\mu)\le (q-1)q^{n-1}
$$
for  $p-1$ values of $\lambda$ and at most $q-1$ values of $\mu$.

We now consider $\lambda^{p}\ne \lam$ and $\lambda^{p^n}=\lam$. We notice that for $1\le j<n$ with $\gcd(j,n)=1$, one has
$$
\gcd\(p^j-1,p^n-1\)=p^{\gcd(j,n)}-1=p-1;
$$ 
thus, as $\lambda^{p-1}\ne 1$, one also has $\lambda^{p^j-1}\ne 1$ and $a_j=0$ by~\eqref{eq:cond1}. In this case, we get 
\begin{equation*}
\begin{split}
M_n(\lambda,\mu)
&\le \left\{\begin{array}{ll}
(q-1)q^{n-\varphi(n)} & 
\text{if $\mu = 0$,}
\\
(q-1)q^{n-1-\varphi(n)} &
\text{if $\mu \ne 0$}.
\end{array}
\right.
\end{split}
\end{equation*}
Since $\lambda^{p^n-1}=1$ and $\lambda^{p-1}\ne 1$, the element $\lambda$ can take at most $\gcd(p^n-1,q-1)-p+1<q$ values.  

Putting everything together, we obtain the bound
$$
L_n(q)< 2(p-1)q^{n-1}+2q^{n-\varphi(n)},
$$
which completes the proof.
\end{proof}

\subsection{Explicit formulas}
Although the general case of polynomials has been studied in \cite{KLMMSS} and in Theorem~\ref{thm:sparse}, it is still worth studying some cases related to special kinds of polynomials. 
Here, for some special kinds of polynomials over $\F_q$ (like linear and power maps), we get explicit formulas for the total number of corresponding non-equivalent dynamical systems. 

Some of these results are straight-forward and  probably well-known, but we give them just 
for completeness of the presentation and to exhibit different types of behaviour.

First, we remark that for a permutation polynomial $f\in\F_q$, every point of $\F_q$ is periodic, and thus, the structure of $\cG_f$ is determined completely by its cycle structure. 

As we know, linear congruential generator and power generator are two classical and simple ways to generate pseudorandom numbers. 
Their cycle structure was extensively studied in~\cite{ChouShp,KurPom,MartPom,ShaHu,SoKr,VaSha} and references therein. The following two theorems suggest that there are not too many such generators up to equivalence.

For linear congruential generator, it is very well known that the cycle structure is completely determined by the distribution of the orders of elements of $\F_q$. 
The next result should be well-known, but for the convenience of the reader (or for the completeness), we present a proof.
\begin{theorem}
\label{thm:lin}
The number of non-equivalent dynamical systems over $\F_q$ generated by the polynomials $f(X)=aX+b, a\in \F_q^*,b\in \F_q$, is equal to $\tau(q-1)+1$, where $\tau(q-1)$ is the number of distinct positive divisors of $q-1$.
\end{theorem}
\begin{proof}
We first consider the dynamical system generated by $f(X)=aX, a\in \F_q^*$. Since for any integer $n\ge 1$ we have $f^{(n)}(X)=a^nX$, it is easy to see that $\cG_f$ has only one fixed point (that is 0) and $(q-1)/m$ cycles of length $m$, where $m$ is the multiplicative order of $a$ in $\F_q^*$ ($m$ divides $q-1$). 

Now, consider $f(X)=aX+b, a\in \F_q^*,b\in \F_q^*$.
Let $\psi$ be the automorphism of $\F_q$ defined by $\psi(X)=bX$. 
Then, we have $\psi^{-1} \circ f \circ \psi = aX+1$. 
Thus, we only need to consider the dynamical system generated by $g(X)=aX+1$. It is also straightforward to see that $\cG_g$ has no fixed point if $a=1$ and otherwise it has only one fixed point (that is $1/(1-a)$) and $(q-1)/m$ cycles of length $m$, where $m$ is the multiplicative order of $a$ in $\F_q^*$ ($m$ divides $q-1$). 

Finally, we conclude the proof by collecting the above results.
\end{proof}

To give a taste of the result in Theorem~\ref{thm:lin}, we indicate that for the divisor function $\tau$, which counts the number of positive divisors of an integer, it is well-known that 
$$
\tau(n)= o(n^{\epsilon})
$$
for any integer $n\ge 1$ and any $\epsilon >0$; for example see  \cite[Formula (31), page 296]{Apostol}. 
In particular, we note that $\tau(n)$ can vary from $2$ to $2^{\log(n)/\log\log(n)}$ for highly composite $n$, see~\cite[Theorem 13.12]{Apostol}.

From the proof of Theorem~\ref{thm:lin}, if  $a$ is a primitive element of $\F_q$ (that is the multiplicative order of $a$ is $q-1$), then the corresponding graph $\cG_f$ only has two cycles, one of length 1 and the other of length $q-1$.
 
Even if the cycle structure of the power generator has been actively studied in~\cite{ChouShp,KurPom,MartPom,ShaHu,SoKr,VaSha} and references therein, the number of distinct functional graphs defined by such maps seems not to have been studied, and thus we present such a result here. 

\begin{theorem}
\label{thm:power}
For a fixed integer $d\ge 1$, the number of non-equivalent dynamical systems over $\F_q$ generated by the polynomials $f(X)=aX^d, a\in \F_q^*$, is equal to $\tau(\gcd(d-1,q-1))$.
\end{theorem}
\begin{proof}
Given $f(X)=aX^d, a\in \F_q^*$, for any integer $n\ge 1$, we have 
$$
f^{(n)}(X)=a^{1+d+\cdots + d^{n-1}}X^{d^n}.
$$
So, the structure of $\cG_f$ is determined completely by its cycle structures. 

Let $\al$ be a primitive element of $\F_q^*$. So, there exists a positive integer $e(a)$ such that $a=\al^{e(a)}$. 
Then, there exists $x\in \F_q^*$ and integer $n\ge 1$ such that $f^{(n)}(x)=x$ if and only if 
the equation 
\begin{equation}
\label{eq:power eq}
a^{1+d+\cdots + d^{n-1}}X^{d^n-1}=1
\end{equation}
has solution in $\F_q^*$, 
which is equivalent to that the equation 
\begin{equation}
\label{eq:lin eq}
(d^n-1)Y+ e(a)(1+d+\cdots + d^{n-1}) \equiv 0 \quad \textrm{(mod $q-1$)}
\end{equation}
with variable $Y$ has solution. It is well-known that the equation \eqref{eq:lin eq} has solution if and only if 
$$
\gcd(d^n-1,q-1) \mid e(a)(1+d+\cdots + d^{n-1}),
$$
that is
$$
\gcd\(d-1,\frac{q-1}{\gcd(1+d+\cdots + d^{n-1},q-1)}\) \mid e(a).
$$
Since 
\begin{align*}
1+d+ \cdots + d^{n-1} & =1+((d-1)+1) + \cdots + ((d-1)+1)^{n-1} \\
&=n+s(d-1)
\end{align*}
for some integer $s$, 
 for any positive integer $t$ satisfying $t \mid \gcd(d-1,q-1)$, we get that $t \mid \gcd(n+s(d-1),q-1)$  if and only if $t \mid \gcd(n,q-1)$. 
 This implies that 
 $$
 \gcd\(d-1,\frac{q-1}{\gcd(n+s(d-1),q-1)}\) =
 \gcd\(d-1,\frac{q-1}{\gcd(n,q-1)}\). 
 $$
  Thus, the above equivalent condition becomes 
 $$
\gcd\(d-1,\frac{q-1}{\gcd(n,q-1)}\) \mid e(a), 
$$
which coincides with 
$$
\gcd\(d-1,\frac{q-1}{\gcd(n,q-1)}\) \mid \gcd(d-1,q-1,e(a)). 
$$
Moreover, if the equation \eqref{eq:lin eq} has solution, then there are exactly $\gcd(d^n-1,q-1)$ solutions modulo $q-1$ (for example see \cite[Proposition 3.3.1]{Ireland}), and thus the equation \eqref{eq:power eq} has exactly $\gcd(d^n-1,q-1)$ solutions. 
Hence, the cycle structures of $\cG_f$ depend only on $\gcd(d-1,q-1,e(a))$. 

Notice that when $n$ tends to infinity, the term $\gcd\(d-1,\frac{q-1}{\gcd(n,q-1)}\)$ can run through all the factors of $\gcd(d-1,q-1)$. Hence, the number of non-equivalent dynamical systems generated by the polynomials $f(X)=aX^d, a\in \F_q^*,$ is equal to $\tau(\gcd(d-1,q-1))$. 
\end{proof}

Recall that $q=p^k$. We also recall the norm function $\Nm_{\F_q/\F_p}(x)=x^{1+p+\cdots+p^{k-1}}$ and the trace function $\Tr_{\F_q/\F_p}(x)=x+x^p+\cdots+x^{p^{k-1}}$ for any $x\in\F_{q}$. The well-known Hilbert's Theorem 90 says that for any $x\in \F_q$, 
\begin{align}
\label{Hilbert90}
& \textrm{$\Nm_{\F_q/\F_p}(x)=1$ if and only if $x=z/z^p$ for some $z\in \F_q$,} \\
& \textrm{$\Tr_{\F_q/\F_p}(x)=0$ if and only if $x=z-z^p$ for some $z\in \F_q$.} \notag
\end{align}

The following result sounds interesting.

\begin{theorem}
\label{thm:power p}
There are only two non-equivalent dynamical systems over $\F_q$ generated by the polynomials $f(X)=aX^p+b, a \in \F_q^*,b\in \F_q$ with $\Nm_{\F_q/\F_p}(a)=1$, depending on whether $\cG_f$ has fixed point or not. In particular, if  $\cG_f$ has fixed point, then it has precisely $p$ fixed points. 
\end{theorem}
\begin{proof}
First, we note that for any $a,b\in \F_q$, $a\ne 0$, the polynomial $f(X)=aX^p+b$ defines naturally a bijection from $\F_q$ to itself. Thus, all the elements of $\F_q$ are periodic points of $\cG_f$. 

Under the assumption $\Nm_{\F_q/\F_p}(a)=1$, by \eqref{Hilbert90} there exists $z\in \F_q$ such that $a=z/z^p$. Defining an automorphism $\psi$ as $\psi(X)=zX$, we have  
$$
\psi^{-1} \circ f \circ \psi (X) = X^p + z^{-1}b.
$$
So, we only need to consider the polynomials $f_b(X)=X^p + b, b\in \F_q$. 

For any integer $n\ge 1$, we have 
$$
f_{b}^{(n)}(X)=X^{p^n}+ b^{p^{n-1}}+\cdots + b^p + b.
$$
Suppose that there exist integer $m \ge 1$ and $x\in \F_q$ such that 
$f_{b}^{(m)}(x)=x$. Then, for any solution $y$ of the equation $X^{p^m}=X$, we have 
$f_{b}^{(m)}(x+y)=x+y$; actually this runs over all the elements of $\F_q$ satisfying $f_{b}^{(m)}(X)=X$. Thus, the number of vertices of $\cG_{f_b}$ with period dividing $m$ is exactly $p^{\gcd(m,k)}$. 

In addition, note that 
$$
f_{b}^{(k)}(X)=X^{q}+ b^{p^{k-1}}+\cdots + b^p + b = X^q + \Tr_{\F_q/\F_p}(b).
$$
We obtain 
$f_{b}^{(kp)}(X)=X^{q^p}$, and thus for any $x\in \F_q$ we have 
$$
f_{b}^{(kp)}(x)=x.
$$
So, for any cycle length $m$ of $\cG_{f_b}$, we have $m \mid kp$. 

By \eqref{Hilbert90}, $\cG_{f_b}$ has fixed point if and only if $\Tr_{\F_q/\F_p}(b)=0$. Since there exists $z\in \F_q$ such that $b=z-z^p$, we define an automorphism $\psi$ as $\psi(X)=X+z$ and derive that 
$$
\psi^{-1} \circ f_b \circ \psi = X^p, 
$$
which means that all these polynomials $f_b$ with $\Tr_{\F_q/\F_p}(b)=0$ generate the same functional graph. Clearly, this graph has precisely $p$ fixed points. 

Now, we consider polynomials $f_b$ with $\Tr_{\F_q/\F_p}(b)\ne 0$. 
Write $k$ as $k=p^e r$ with integer $e\ge 0$ and $\gcd(r,p)=1$, 
and let $c(b)$ be the smallest cycle length of $\cG_{f_b}$. 
Notice that for any $x\in \F_q$ we have $f_{b}^{(k)}(x)=x + \Tr_{\F_q/\F_p}(b)$. So,  there is no element $x\in \F_q$ such that $f_{b}^{(k)}(x)=x$, and thus  $c(b) \nmid k$. On the other hand, we have known that the number of vertices of $\cG_{f_b}$ with period $c(b)$ is exactly $p^{\gcd(c(b),k)}$, 
which implies that $c(b)$ is some power of $p$. 
Noticing $c(b) \mid kp$, we must have $c(b)=p^{e+1}$. 
We can also see that any cycle length $m$ of $\cG_{f_b}$ has the form $m=p^{e+1}s$ with some integer $s\mid r$ (because $m \mid kp$ and $m \nmid k$), and so $c(b) \mid m$. Then, by the discussion in the third paragraph, for any cycle length $m$, the number of vertices of $\cG_{f_b}$ with period dividing $m$ is exactly $p^{\gcd(m,k)}$,  which is independent of $b$. Thus, all such polynomials $f_b$ with $\Tr_{\F_q/\F_p}(b)\ne 0$ generate the same functional graph. This concludes the proof. 
\end{proof}

In fact, we can get more general result. 

\begin{theorem}
\label{thm:power p2}
The number of non-equivalent dynamical systems over $\F_q$ generated by the polynomials $f(X)=aX^p+b, a \in \F_q^*,b\in \F_q$ is equal to $\tau(p-1)+1$. In particular, there is only one such system up to equivalence having no fixed point. Moreover, if $\cG_f$ has fixed point, then it has exactly $p$ fixed points if $\Nm_{\F_q/\F_p}(a)=1$, and otherwise it has only one fixed point. 
\end{theorem}

\begin{proof}
For $f(X)=aX^p+b, a \in \F_q^*,b\in \F_q$, we first suppose that $\cG_f$ has a fixed point. That is, there exists $z\in \F_q$ such that $az^p + b = z$. Then, defining an automorphism $\psi$ as $\psi(X)=X+z$, we get 
$$
\psi^{-1} \circ f \circ \psi = aX^p. 
$$
Thus, by Theorem \ref{thm:power}, the number of these systems up to equivalence is exactly $\tau(p-1)$. For such $\cG_f$, the number of its fixed points can also be easily obtained. 

Now, suppose that $\cG_f$ has no fixed point. 
That is, for any $\mu \in \F_q$ we have $a\mu^p+b-\mu \ne 0$. We fix one $\mu$ and put $\lam = a\mu^p+b-\mu$. Defining an automorphism $\psi$ as $\psi(X) = \lam X + \mu$, we obtain 
$$
\psi^{-1} \circ f \circ \psi = a \lam^{p-1}X^p + 1. 
$$
Thus, we only need to consider polynomials $f_a(X) = a X^p + 1, a \in \F_q^*,$ such that $\cG_{f_a}$ has no fixed point. For these polynomials $f_a$, assume that $\Nm_{\F_q/\F_p}(a) \ne 1$, which will lead to a contradiction.  
Indeed, we have 
$$
f_a^{(k)}(X) = \Nm_{\F_q/\F_p}(a) X^q + a^{1+p+\cdots + p^{k-2}} + \cdots + a +1. 
$$
Under the assumption $\Nm_{\F_q/\F_p}(a) \ne 1$, the equation $f_a^{(k)}(X)=X$ has  only one solution, say $y$, in $\F_q$ (note that $x^q=x$ for any $x\in \F_q$). So, $y$ must be a fixed point of $\cG_{f_a}$. This contradicts with the fact that $\cG_{f_a}$ has no fixed point. Thus, we must have $\Nm_{\F_q/\F_p}(a)=1$. Then, the desired result follows from Theorem \ref{thm:power p} and the above discussion. 
\end{proof}

It may deserve stating the following as a separate result. 

\begin{corollary}
\label{cor:power p}
The number of non-equivalent dynamical systems over $\F_q$ generated by the polynomials $f(X)=aX^p+b, a \in \F_q^*,b\in \F_q$ with $\Nm_{\F_q/\F_p}(a)\ne 1$ is equal to $\tau(p-1)-1$. In particular, these dynamical systems are not equivalent to those in Theorem \ref{thm:power p}.
\end{corollary}

\begin{proof}
By Theorems \ref{thm:power p} and \ref{thm:power p2}, we only need to prove that for any polynomial $f(X)=aX^p+b, a \in \F_q^*,b\in \F_q$ with $\Nm_{\F_q/\F_p}(a)\ne 1$, the functional graph $\cG_f$ has fixed point. 
For such a polynomial $f$, assume that $\cG_f$ has no fixed point. Then, as in the proof of Theorem \ref{thm:power p2}, we see that there exists $\lam \in \F_q^*$ such that $\cG_f$ is isomorphic to $\cG_g$, where 
$$
g(X) = a\lam ^{p-1} X^p +1. 
$$ 
Note that 
$$
\Nm_{\F_q/\F_p}(a\lam^{p-1})=\Nm_{\F_q/\F_p}(a)\Nm_{\F_q/\F_p}(\lam)^{p-1}=\Nm_{\F_q/\F_p}(a) \ne 1, 
$$
which as before leads to a contradiction.  
This completes the proof. 
\end{proof}

\subsection{The case of rational functions}

For any rational function $f/g$, where $f,g\in \F_q[X]$, we can define a dynamical system over $\F_q$ as follows
\begin{equation}
\label{def:system}
\begin{split}
f/g: \F_q \to \F_q, \quad x 
&\mapsto\left\{\begin{array}{ll}
f(x)/g(x),& \text{if $g(x)\ne 0$,}\\
\al,&\text{otherwise,}
\end{array}
\right.
\end{split}
\end{equation}
where $\al \in \F_q$ is fixed. Besides, for the rational function $f/g$ we can define a dynamical system over the projective line $\PP^1(\F_q)=\F_q \cup \{\infty\}$ in the natural way: 
\begin{equation}
\label{def:system2}
f/g: \PP^1(\F_q) \to \PP^1(\F_q), \quad x \mapsto f(x)/g(x),
\end{equation}
where every pole of $f/g$ (after clearing common factors) is mapped to infinity.

In this section, we first estimate the number of non-equivalent dynamical systems over $\F_q$ generated by rational functions with the form \eqref{def:system}. Then, we indicate that these estimates are also valid for such systems defined by \eqref{def:system2}.  

For non-negative integers $m,n$, define
\begin{align*}
S_{m,n}(q)= \{ f/g: f(X),g(X) \in \F_q[X], & \deg f=m, \\
& \deg g=n, \textrm{$g$ is monic}\}.
\end{align*}
Note that $S_{m,n}(q)$ is exactly the set consisting of the rational functions of the forms $f/g$, where $f,g\in \F_q[X]$ with $\deg f=m$ and $\deg g=n$. In particular, the set $S_{m,0}(q)$ exactly consists of polynomials of degree $m$, and it has been studied in \cite{KLMMSS}. 

Now, let $N_{m,n}(q)$ be the number of non-equivalent dynamical systems generated by the set $S_{m,n}(q)$ as \eqref{def:system}. 
Since $\# S_{m,n}(q)= (q-1)q^{m+n}$, we have the following trivial upper bound 
$$
N_{m,n}(q) < q^{m+n+1}.
$$
Here, we give a non-trivial upper bound for $N_{m,n}(q)$. 

\begin{theorem}
\label{thm:Nmn1}
For any non-negative integers $m,n$ with $m+n \ge 1$, define two non-negative integers $t,r$ by the Euclidean division
$$
m=t|m-n-1|+r, \quad 0 \le r < |m-n-1|.
$$
Let $r^*$ be the number of integers $i$, $1\le i \le r$, such that $\gcd(i,|m-n-1|)\ne 1$ if $r\ge 1$; otherwise if $r=0$,  let $r^*=0$. 
Then, we have 
\begin{equation*}
\begin{split}
N_{m,n}(q)
&\le \left\{\begin{array}{ll}
q^{m+n}+(s-1)q^{n+t(|m-n-1|-\varphi(|m-n-1|))+r^*} & 
\text{if $m\ge 1$,}
\\
q+s-1 &
\text{if $m=0,n=1$},
\\
q^{n}+(s-1)q^{n-2} &
\text{if $m=0,n\ge 2$},
\end{array}
\right.
\end{split}
\end{equation*}
where $s=\gcd(q-1,|m-n-1|)$, and $\varphi$ is Euler's totient function.
In particular, $N_{m,n}(q)\le q^{m+n}$ if $|m-n-1|=1$, and if $|m-n-1|\ge 2$ we have $N_{m,n}(q)\le 2q^{m+n}$. Furthermore, $N_{m,n}(q)\le q^{m+n}$ if $s=1$.
\end{theorem}

\begin{proof}
We use the same idea as in Theorem~\ref{thm:sparse1}. Indeed, 
for $\lambda \in \F_q^*$, 
the bijections
$$
\psi_{\lam}:\  X \mapsto \lambda X
$$
with inverse $\psi_{\lam}^{-1}: \  X \mapsto \lambda^{-1}X$, 
form a group of order $(q-1)$ in the usual way, which acts on the set $S_{m,n}(q)$ as the map
$$
f(X)/g(X) \to \psi_{\lam}^{-1} \circ f/g \circ \psi_{\lam}(X),
$$
where $f/g \in S_{m,n}(q)$. Generally, we write $f,g$ as 
\begin{align}
\label{eq:f/g}
& f(X)=a_mX^m+a_{m-1}X^{m-1}+ \cdots + a_0,  \\
& g(X)=X^n+b_{n-1}X^{n-1}+ \cdots + b_0. \notag
\end{align}

As before, we have 
\begin{equation} 
\label{Burnside1}
N_{m,n}(q)\le \frac{1}{q-1}\sum_{\lam \in \F_q^*} M_{m,n}(\lam),
\end{equation}
where  $M_{m,n}(\lam)$ is the number of rational functions in $S_{m,n}(q)$  fixed by $\psi_{\lam}$ under the above group action.

Trivially, we have
\begin{equation}
\label{eq:lam=1}
M_{m,n}(1)=(q-1)q^{m+n}.
\end{equation}

For any $f/g \in S_{m,n}(q)$ with the form \eqref{eq:f/g} satisfying
 $\psi_{\lam}^{-1} \circ f/g \circ \psi_{\lam}(X)= f(X)/g(X)$,
 we have
\begin{equation}
\label{eq:equality}
f(\lambda X)g(X) = \lam f(X)g(\lambda X).
\end{equation}
Comparing the leading coefficients we
derive 
$$
a_m(\lam^m-\lam^{n+1})=0,
$$
which implies that 
$\lambda^{m-n-1} = 1$. 
So
\begin{equation}
\label{eq: good lam}
M_{m,n}(\lam)=0
\end{equation}
for any $\lam$ with $\lambda^{m-n-1} \ne 1$.

Assume now that $\lambda^{m-n-1} = 1$ but $\lambda \ne 1$, which implies that $|m-n-1|>1$. Comparing the coefficients of $X^{n+j}$ in  both sides of the equality \eqref{eq:equality},  
we see that for every $j = 0,1, \ldots, m$ there are polynomials
$$
F_j\in \F_q[Y_{j+1},\ldots,Y_m,Z_0,\ldots,Z_{n-1}, U]
$$
such that
$$
a_j(\lam^j - \lam^{n+1}) = F_j(a_{j+1},\ldots, a_m,b_0,\ldots,b_{n-1}, \lam),
$$
where in particular $F_m=0$.
Since $\lam \ne 1$ and $\lam^{m-n-1} = 1$, it follows
that for every $j$, $j=0,1,\ldots, m$, with $\gcd(j-n-1,m-n-1)=1$ we have $\lam^{j} \ne \lam^{n+1}$
and thus $a_j$ is uniquely defined by $a_{j+1},\ldots, a_m,b_0,\ldots,b_{n-1}, \lam$. Note that 
$$
\gcd(j-n-1,m-n-1)=\gcd(m-j,m-n-1).
$$
So, if $m\ge 1$, it is equivalent to count how many integers $i$ ($i=0,1,\ldots,m$) are not coprime to $m-n-1$; thus, for $|m-n-1|>1$ and any  $\lam$ satisfying $\lam^{m-n-1} = 1$ and $\lam \ne 1$, we have
\begin{equation}
\label{eq:lam not 1}
M_{m,n}(\lam)\le (q-1)q^{n+t(|m-n-1|-\varphi(|m-n-1|))+r^*},
\end{equation}
where $m\ge 1$. If $m=0$ (so $n\ge 1$), by \eqref{eq:equality} we obtain 
$$
a_0(X^n+b_{n-1}X^{n-1}+\cdots + b_0) = \lam a_0 \((\lam X)^n+b_{n-1}(\lam X)^{n-1}+\cdots + b_0 \).
$$
As $a_0 \ne 0$, we get 
$$
b_i (\lam^{i+1}-1)=0, \quad i=0,1, \ldots, n-1.
$$
Since in this case $\lam^{n+1}=1$ and $\lam \ne 1$, we must have $b_0=0, b_{n-1}=0$. Thus, if $m=0$, $\lam^{n+1}=1$ and $\lam \ne 1$, we have 
\begin{equation}
\label{eq:m=0}
\begin{split}
M_{m,n}(\lam)
&\le \left\{\begin{array}{ll}
q-1  & \text{if $n=1$,}
\\
(q-1)q^{n-2} &
\text{if $n\ge 2$}.
\end{array}
\right.
\end{split}
\end{equation}
Notice that since  $\lam^{m-n-1} = 1$ and $\lam \ne 1$,  the element $\lam$ can take at most $\gcd(q-1,|m-n-1|)-1$ values.

Using~\eqref{Burnside1} together with~\eqref{eq:lam=1}, \eqref{eq: good lam}, \eqref{eq:lam not 1} and \eqref{eq:m=0}, we complete the proof.
 \end{proof}

Provided $m-n\ge 2$, we can further obtain an improvement. 
Note that if $n=0$, the result has been already given in \cite[Theorem 1]{KLMMSS}.

\begin{theorem}
\label{thm:Nmn2}
For any non-negative integers $m,n$ with $m-n\ge 2$, define two non-negative integers $t,r$ by the Euclidean division
$$
m=t(m-n-1)+r, \quad 0 \le r < m-n-1.
$$
Let $r^*$ be the number of integers $i$, $1\le i \le r$, such that $\gcd(i,m-n-1)\ne 1$ if $r\ge 1$; otherwise if $r=0$,  let $r^*=0$. 
Then, we have
\begin{equation*}
\begin{split}
N_{m,n}(q)
&\le \left\{\begin{array}{ll}
q^{m+n-1}+(s-1)q^{n+t(m-n-1-\varphi(m-n-1))+r^*}\\
\qquad\qquad\qquad\qquad\qquad\qquad\qquad\qquad\quad 
\text{if $p\nmid m-n$,}
\\
q^{m+n-1}+(s-1)q^{n+t(m-n-1-\varphi(m-n-1))+r^*}+(q-1)q^{m/p-1}\\
\qquad\qquad\qquad\qquad\qquad\qquad\qquad\qquad\quad
\text{if $p \mid m-n$},
\end{array}
\right.
\end{split}
\end{equation*}
where $s=\gcd(q-1,m-n-1)$.
In particular, $N_{m,n}(q)\le 2q^{m+n-1}$ if $m-n=2$, and if $m-n\ge 3$ we have $N_{m,n}(q)\le 3q^{m+n-1}$. Furthermore, $N_{m,n}(q)\le q^{m+n-1}$ if $p\nmid m-n$ and $s=1$.
\end{theorem}

\begin{proof}
For $\lambda \in \F_q^*$ and $\mu \in \F_q$,  as in the proof of Theorem~\ref{thm:lin}, we define the bijection 
$\phi_{\lambda,\mu}$ and its inverse $\phi_{\lambda,\mu}^{-1}$. Each bijection $\phi_{\lambda,\mu}$  acts on the set $S_{m,n}(q)$ as the map
$$
f(X)/g(X) \to \phi_{\lambda,\mu}^{-1} \circ f/g \circ \phi_{\lambda,\mu}(X),
$$
where $f/g \in S_{m,n}(q)$. 

 As before, we have
\begin{equation} 
\label{Burnside}
N_{m,n}(q)\le \frac{1}{(q-1)q}\sum_{(\lambda,\mu)} M_{m,n}(\lambda,\mu),
\end{equation}
where the sum runs through all the pairs $(\lambda,\mu)\in \F_q^*\times \F_q$, and $M_{m,n}(\lambda,\mu)$ is the number of rational functions in $S_{m,n}(q)$  fixed by $\phi_{\lambda,\mu}$ under the above group action.

Trivially, we have
\begin{equation}
\label{eq:mu=0}
M_{m,n}(1,0)=(q-1)q^{m+n}.
\end{equation}

In the following, we want to estimate $M_{m,n}(\lambda,\mu)$ by fixing a pair $(\lambda, \mu)\in \F_q^*\times \F_q\setminus\{(1,0)\}$.

For any $f/g \in S_{m,n}(q)$ with the form \eqref{eq:f/g} satisfying
 $\phi_{\lambda,\mu}^{-1} \circ f/g \circ \phi_{\lambda,\mu}(X)= f(X)/g(X)$,
 we have
\begin{equation}
\label{eq:fix form}
f(\lambda X + \mu)g(X) = \lambda f(X)g(\lambda X + \mu) + \mu g(\lambda X + \mu)g(X).
\end{equation}
Comparing the leading coefficients we
derive 
$$
a_m(\lam^m-\lam^{n+1})=0,
$$
which implies that 
$$
\lambda^{m-n-1} = 1.
$$
So
\begin{equation}
\label{eq: good lambda}
M_{m,n}(\lambda,\mu)=0
\end{equation}
for any $(\lambda,\mu)$ not satisfying $\lambda^{m-n-1} = 1$.

First, suppose that $\lambda=1$. Note that $\mu \ne 0$.
Comparing the coefficients of $X^{m+n-1}$ in  both sides of the equality \eqref{eq:fix form}, 
we obtain 
\begin{equation}
\label{eq:m+n-1}
(m-n)a_m \mu = 0.
\end{equation}
Thus, $p\mid (m-n)$, here $p$ is the characteristic of $\F_q$.
Moreover, comparing the coefficients of $X^{n+j-1}$ in  both sides of the equality \eqref{eq:fix form}  for every $j=0,1, \ldots, m$ (in fact, they are sums of several terms, we only need to consider those terms where $a_j$ or $a_{j-1}$ appear (if $j=0$, only $a_0$), and we don't need to consider the coefficients in $\mu g(\lambda X + \mu)g(X)$),
we also obtain
relations of the form
$$
(j-n)a_{j}\mu= F_j(a_{j+1},\ldots, a_m,b_0,\ldots,b_{n-1}, \mu), 
\quad j=0,1,\ldots, m,
$$
for some polynomials
$$
F_j\in \F_q[Y_{j+1},\ldots,Y_m,Z_0,\ldots,Z_{n-1}, U],
$$
where in the case $j=m$ we have $F_m = 0$, which corresponds to~\eqref{eq:m+n-1}.
In particular, for every $j=0,1, \ldots, m$ with $\gcd(j-n,p)=1$,
we see that $a_j$ is uniquely defined by $a_{j+1}, \ldots, a_m,b_0,\ldots,b_{n-1} , \mu$. Notice that when $p \mid (m-n)$, we have 
$$
\gcd(j-n,p)=\gcd(m-n-(j-n),p)=\gcd(m-j,p).
$$
Hence, for $\mu \ne 0$ we get that
\begin{equation}
\begin{split}
\label{eq:lambda =1}
M_{m,n}(1,\mu)
&\le \left\{\begin{array}{ll}
0,& \text{if $p\nmid (m-n)$,}\\
(q-1)q^{m/p},&\text{if $p \mid (m-n)$}.
\end{array}
\right.
\end{split}
\end{equation}

Assume now that $\lambda^{m-n-1} = 1$ but $\lambda \ne 1$, which implies that $m-n\ge 3$. As the above, comparing the coefficients of $X^{n+j}$ in  both sides of the equality \eqref{eq:fix form}  for every $j=0,1, \ldots, m$, 
we see that there are polynomials
$$
G_j\in \F_q[Y_{j+1},\ldots,Y_m,Z_0,\ldots,Z_{n-1}, U,V]
$$
such that
$$
a_j(\lam^j - \lam^{n+1}) = G_j(a_{j+1},\ldots, a_m,b_0,\ldots,b_{n-1}, \lam,\mu),
$$
where in particular $G_m=0$.
Since $\lam \ne 0,1$, and $\lam^{m-n-1} = 1$, it follows
that for every $j$, $j=0,1,\ldots, m$, with $\gcd(j-n-1,m-n-1)=1$ we have $\lam^{j} \ne \lam^{n+1}$
and thus $a_j$ is uniquely defined by $a_{j+1},\ldots, a_m,b_0,\ldots,b_{n-1}, \lam,\mu$. So as before, it is equivalent to count how many integers $i$ ($i=0,1,\ldots,m$) are not coprime to $m-n-1$. Thus, for $m-n\ge 3$ and any pair $(\lam,\mu)$ satisfying $\lam^{m-n-1} = 1$ and $\lam \ne 1$, we have
\begin{equation}
\label{eq:lambda not 1}
M_d(\lambda,\mu)\le (q-1)q^{n+t(m-n-1-\varphi(m-n-1))+r^*}.
\end{equation}
Notice that since  $\lam^{m-n-1} = 1$ and $\lam \ne 1$,  the element $\lam$ can take at most $\gcd(q-1,m-n-1)-1$ values.

Using~\eqref{Burnside} together with~\eqref{eq:mu=0}, \eqref{eq: good lambda}, \eqref{eq:lambda =1}
and~\eqref{eq:lambda not 1}, we complete the proof.
 \end{proof}
 
\begin{remark}
In the proofs of Theorems \ref{thm:Nmn1} and \ref{thm:Nmn2}, we actually classify rational functions under the action of affine automorphisms. 
So, the results are also true for such dynamical systems generated by  corresponding rational functions as \eqref{def:system2}.  
\end{remark}

 \section{Comments and questions}

In this paper, we only study the total number of dynamical systems up to equivalence. In practice, some kinds of dynamical systems with prescribed properties are preferable depending on applications. So, it is meaningful and also interesting to prove the existence and estimate the amount of some special kinds of dynamical systems. For example, when using a polynomial $f(X)$ over $\F_p$ to produce pseudorandom numbers, we prefer that $\cG_f$ has a cycle of large length. Here, we mention some 
questions of Shparlinski (personal correspondence).

\begin{question}
Tests show that for any prime $p$ there is a polynomial $f(X) = X^2 +a \in \F_p[X]$ such that $\cG_f$ has only one component. Can we prove this? 
What about polynomials of higher degree? Moreover, how many distinct graphs $\cG_f$ having only one component are there for every prime $p$?
\end{question}

Carlitz~\cite{Carl} (see also~\cite{Zieve}) has proved the following fundamental result. 
For $q>2$, all permutation polynomials over $\F_q$ can be generated by 
the following two classes of permutation polynomials,
\begin{equation*}
  aX+b,\quad a,b\in\F_q,\ a\ne0\text{ and } X^{q-2}.
\end{equation*}
Thus, 
every permutation polynomial of $\F_q$ can be
represented by 
\begin{equation}
\label{eq:Pk}
  P_{k}(X)=\(\ldots
  \((a_0X+a_1)^{q-2}+a_2\)^{q-2}+\ldots+a_{k}\)^{q-2}+a_{k+1}, 
\end{equation}
with some integer $k$, 
where $a_1,a_{k+1}\in\F_q$, $a_i\in\F_q^*$, $i=0,2,\ldots,k$, see~\cite{CMT} for more details.

The authors of~\cite{ACMT} define the {\it Carlitz rank\/} of a 
permutation polynomial $f$ over $\F_q$ to be the smallest positive
integer $k$ satisfying $f=P_{k}$ for a permutation $P_{k}$ of the 
form~\eqref{eq:Pk}, and denote it by $\Crk f$.
In other words, $\Crk f=k$ if $f$ is a composition of at least $k$ inversions
$X^{q-2}$ and $k+1$ (or $k$ if $a_{k+1} = 0$)  linear polynomials. 

As mentioned above, for a permutation polynomial all points are periodic, and thus the functional graph is determined by the cycle structure. 
It would be certainly interesting to give lower or upper bounds for the number of non-isomorphic functional graphs defined by permutation polynomials of Carlitz rank at most $k$ (or exactly $k$).

\section*{Acknowledgements}

The authors want to thank the referee for careful reading and valuable comments.  They would like to thank Igor E. Shparlinski for his useful  suggestions and stimulating discussions.
They are also grateful to the Max Planck
Institute for Mathematics in Bonn, for hosting them during
the program ``Dynamics and Numbers". 

The research of A.~O. was supported by the 
UNSW Vice Chancellor's Fellowship and of M.~S. by the Australian Research Council Grant DP130100237.

\end{document}